\documentclass[preprint,12pt]{elsarticle}
\usepackage[utf8]{inputenc}
	
\usepackage{amssymb, amsmath, amsthm}

\usepackage{thmtools} 
\usepackage{thm-restate}

\usepackage{hyperref}
\usepackage[backgroundcolor=yellow]{todonotes}

\newtheorem{theorem}{Theorem}[section]

\newtheorem{example}[theorem]{Example}

\newtheorem{proposition}[theorem]{Proposition}
\newtheorem{corollary}[theorem]{Corollary}

\newtheorem{definition}[theorem]{Definition}
\newtheorem{remark}[theorem]{Remark}

\newcommand{\F}{\mathcal F}
\newcommand{\G}{\mathcal G}
\newcommand{\mS}{\mathcal{S}}
\newcommand{\K}{\mathcal{K}}

\newcommand{\Hd}{\mathcal{H}}

\newcommand{\I}{\mathcal I}
\newcommand{\N}{\mathbb N}

\newcommand{\R}{\mathbb R}

\newcommand{\on}{\operatorname}

\begin{document}

\journal{J Math Anal Appl}

\begin{frontmatter}

\title{On a typical compact set as the attractor of generalized iterated function systems of infinite order}


\author{{\L}ukasz Ma\'slanka}
\address{Institute of Mathematics, Polish Academy of Sciences, \'Sniadeckich 8, 00-656 Warszawa, Poland}
\ead{lmaslanka@impan.pl}

\begin{abstract}
In 2013 Balka and M\'ath\'e showed that in uncountable polish spaces the typical compact set is not a fractal of any IFS. \\
In 2008 Miculescu and Mihail introduced a concept of a \emph{generalized iterated function system} (GIFS in short), a particular extension of classical IFS, in~which they considered families of mappings defined on finite Cartesian product $X^m$ with values in $X$.  
Recently, Secelean extended these considerations to mappings defined on the space $\ell_\infty(X)$ of all bounded sequences of elements of $X$ endowed with supremum metric. \\
In the paper we show that in Euclidean spaces a typical compact set is an~attractor in sense of Secelean and that in general in the polish spaces it can be perceived as selfsimilar in such sense.
\end{abstract}


\end{frontmatter}


\section{Introduction}
Let $(X,d)$ be a metric space. By $\K(X, d)$ (shortly $\K(X)$) we denote the space of all nonempty and~compact subsets of $X$, endowed with the~Hausdorff-Pompeiu metric
$$H^d(K,D):=\max\left\{\sup\{\inf\{d(x,y):x\in K\}:y\in D\},\sup\{\inf\{d(x,y):x\in D\}:y\in K\}\right\}.$$
If $d$ is fixed, we will write $H:=H^d$. It is well known that $H^d$ induces the topology consistent with \emph{Vietoris topology} on $\K(X)$, i.e. the topology which base consists of nonempty sets 
$$U[U_1, ..., U_n] := \left\{K \in \K(X): K \subset U, \ K \cap U_i \neq \emptyset \right\}$$
where $n\in\N$ and $U, U_1, ..., U_n \subset X$ are open. \\
Moreover, $\K(X)$ is complete (compact), provided $X$ is complete (compact). 

Let $(X,d), (Y, \rho)$ be metric spaces and $f:X \to Y$. By $\on{Lip}(f)$ we denote the Lipschitz constant of $f$. We say that $f$ is a \emph{contraction} if $\on{Lip}(f)<1$ and that it is a \emph{weak contraction} if $\rho(f(x), f(y)) < d(x,y)$ provided $x \neq y$. We say that a family $\F:=\{f_1, ..., f_n\}$ is an \emph{iterated function system} (\emph{IFS} in~short) if $f_i: X \to X, i=1,...,n,$ are contractions. We call $\F$ a \emph{weak IFS} if~$f_i: X \to X$ are weak contractions. 

The classical Hutchinson-Barnsley theorem (\cite{Hutchinson1981},\cite{Barnsley1993}) states that if $X$ is a~complete metric space and $\F$ is an IFS then there exists a unique \mbox{$A_\F \in \K(X)$} such that 
$$A_\F = \bigcup_{i=1}^n f_i(A_\F).$$
The set $A_\F$ is called the \emph{fractal} or the \emph{attractor of the IFS $\F$}. Sometimes it is also called \emph{self-similar set}. 

In 2008 some interesting variation on the notion of an IFS was considered. More precisely, Mihail and Miculescu (see \cite{Mihail2008}, \cite{Mihail2008b}, \cite{Mihail2010}) introduced the notion of a \emph{generalized iterated function systems} (\emph{GIFSs} in short), in which they considered families of contractions defined on a finite Cartesian product~$X^m$ (equipped with the maximum metric) of a metric space $X$ with its values in $X$. They obtained a counterpart of the H-B theorem which states that if $X$ is complete and $\G := \{g_1, ..., g_n\}$ is a GIFS, then there exists a unique $A_\G \in \K(X)$ such that
\begin{equation}\label{gifsG}
A_\G = \bigcup_{i=1}^n g_i(A_\G \times ... \times A_\G).
\end{equation}
$A_\G$ is called the \emph{fractal} or the \emph{attractor of a GIFS $\G$}. (\ref{gifsG}) shows that GIFSs' fractals are also self-similar, but in more complex way.

Another step in an attempt to describe sets in the language of self-similarity was done by Secelean who considered families of mappings defined on the space $\ell_\infty(X)$ of all bounded sequences in $X$ (endowed with supremum metric) with values in $X$. A finite family of contractions from $\ell_\infty(X)$ to $X$ we will call a \emph{generalized iterated function system of order infinity} (\emph{GIFS$_\infty$} in short). With some extra, technical assumption (C1) (which we will discuss later) he obtained some counterpart of the H-B theorem, namely: If $X$ is complete and $\mS$ is a GIFS$_\infty$ satisfying (C1), then there exists a unique $A_\mS \in \K(X)$ such that
\begin{equation}\label{attgifs}
A_\mS = \bigcup_{g \in \mS} \overline{g (A_\mS \times A_\mS \times ...)}.
\end{equation}
We call $A_\mS$ the \emph{fractal} or the \emph{attractor of~GIFS$_\infty$ $\mS$}. The self-similarity of GIFSs$_\infty$' fractals is even more advanced and may not be completely precise (sometimes we can omit taking closures in (\ref{attgifs})). 

To obtain his result, Secelean used some iteration procedure which is~not a~very natural counterpart of the one used by Mihail and Miculescu. Trying to receive such a natural counterpart, Strobin and the~author considered the~waged supremum metric $d_{q}$ on $\ell_\infty(X)$ defined for $q\leq 1$ by 
\begin{equation}
d_{q}((x_n), (y_n)) := \sup_{n\in\N} q^{n-1} d(x_n, y_n), \ \ \ (x_n), (y_n) \in \ell_\infty(X).
\end{equation}
It appears that under some more rigorous conditions, attractors of GIFSs$_\infty$' have better properties (see~\cite{Maslanka2018}).

In 2003 Balka and M\'ath\'e (\cite{Balka2013}, \cite{Balka2013a}) with the use of \emph{generalized Hausdorff measure} proved that in uncountable polish spaces, the typical compact set is not an attractor of any weak IFS. Also, using different techniques D'Aniello and Steele (\cite{DAniello2015}) proved the same claim for Euclidean spaces. It is also worth underlying that in Hilbert spaces a typical compact set is not even an attractor of any GIFS (see \cite[Theorem 7.1]{Maslanka2019}). We will modify slightly a construction of Balka and M\'ath\'e to prove that the typical (in~the Baire category sense) compact set can be perceived as \emph{self-similar in the sense of Secelean} (i.e. it satisfies (\ref{attgifs}) for some GIFS$_\infty$) and that in Euclidean spaces it is the attractor of some GIFS$_\infty$. 

It is worth noting that the topological properties of fractals are investigated from many points of view. In particular there appears a notion of a~\emph{topological IFS fractal} (see \cite{Kameyama2000}, \cite{Mihail2012}; in \cite{Kameyama2000} it is called a self similar set), i.e. a compact Hausdorff space $X$ such that for some finite family $\F$ of continuous selfmaps of $X$, $X=\bigcup_{f\in\F}f(X)$ and for every sequence $(f_k)\subset \F$, the set 
$$
\bigcap_{k\in\N}f_{1}\circ...\circ f_{k}(X)
$$
(called sometimes a \emph{fibre}), is singleton. As was proved in \cite{Banakh2015} and \cite{Miculescu2015}, $X$~is a~topological IFS fractal iff $X$ is homeomorphic to the attractor of some weak IFS (in particular, it is metrizable). 


In the next section we recall the frameworks of generalized IFSs -- GIFSs and GIFSs$_\infty$. In particular we explain the (C1) condition. We recall counterparts of H-B theorem obtained by the mentioned authors. 
Then in Section 3 we introduce some basic notions and definitions. 
Section 4 is devoted to introducing a slight modification of Balka and M\'ath\'e construction of so called \emph{balanced set}. We also gather results obtained by them. 
Then, in Section 5,  we prove the main theorem which states that in the Polish spaces a typical compact set is self-similar in the sense of Secelean. Also we prove that in~Euclidean spaces it is GIFS$_\infty$' fractal. 
Finally, in the Appendix we correct a proof given by Balka and M\'ath\'e in \cite{Balka2013a} (the theorem, see Theorem~\ref{twzbal}, states that in Polish spaces a typical compact set is either finite or it is a union of~balanced set and finite set). 

\section{Generalizations of notion of iterated function systems}\label{section2}
\subsection{Generalized iterated function systems of finite orders}
Let $m\in\N$ and $X^m$ be the Cartesian product of $m$ copies of $X$, endowed with the maximum metric. A finite family $\G=\{g_1,...,g_n\}$ of contractions $g_i:X^m\to X$ is called a \emph{generalized iterated function systems of order~$m$} (GIFS in short). 
Miculescu and Mihail in \cite{Mihail2008} and \cite{Mihail2010} (also \cite{Mihail2008b} for the case of~compact~$X$) proved the following version of the H-B theorem:
\begin{theorem}\label{gifs}
Let $X$ be a complete metric space, $m\in\N$ and $\G=\{g_1,...,g_n\}$ be a GIFS of order $m$. Then there is a unique set $A_\G\in\K(X)$ such that
\begin{equation}\label{ffinal1}
A_\G=\bigcup_{i=1}^n g_i(A_\G\times...\times A_\G).
\end{equation}
\end{theorem}
\noindent The set $A_\G$ is called the \emph{fractal} or the \emph{attractor} of GIFS $\G$.

This result was extended by Strobin and Swaczyna in \cite{Strobin2013} to mappings which satisfy weaker contractive conditions. The results from articles \cite[Example 4.3]{Mihail2010} and \cite{Strobin2015} (a construction of a Cantor-type GIFS' fractal which is not a fractal of any IFS) point out that class of GIFSs' fractals is essentially wider than class of IFSs' fractals. Moreover results from \cite{Maslanka2019} combined with~\cite{Nowak2013} shows that there are some GIFSs' fractals which are not fractals of any weak IFS. 

\subsection{Generalized iterated function systems defined on the $\ell_\infty$-sum -- Secelean's approach}
Secelean in \cite{Secelean2014} considered mappings defined on $\ell_\infty$-sum of a space $X$ (mainly with supremum metric) with values in~$X$. We will restrict our attention to contractions (although Secelean dealt with more general contractive conditions).  
Given a metric space $(X,d)$, let $\ell_\infty(X)$ be the $\ell_\infty$-sum of $X$, i.e.,
\begin{equation*}
\ell_\infty(X):=\{(x_n)\subset X:(x_n)\;\mbox{is bounded}\}.
\end{equation*}
We endow $\ell_\infty(X)$ with supremum metric $d_1$. 

\begin{remark}\emph{
Clearly, if $X$ is bounded then $\ell_\infty(X)$ is just the Cartesian product. 
}
\end{remark}

Following \cite{Secelean2014}, we say that $f:\ell_\infty(X) \to X$ satisfies the condition \emph{(C1)} if
\begin{center}
$(C1)\;\;\;$ for every $(K_k)\in\ell_\infty(\K(X))$, the closure of the image of the product $\overline{f(\prod_{k=1}^\infty K_k)}\in\K(X)$,
\end{center}
and condition \emph{(C2)}, if
\begin{center}
$(C2)\;\;\;$ for every $(K_k)\in\ell_\infty(\K(X))$, the image of the product ${f(\prod_{k=1}^\infty K_k)}\in\K(X).\;\;\;\;\;\;\;\;\;\;\;\;\;\;\;\;\;\;\;\;\;\;\;$
\end{center}

\begin{remark}\emph{
Observe that since $(K_k) \in \ell_\infty(\K(X))$ the sum $\bigcup_{k=1}^\infty K_k \subset X$ is bounded and hence $f\left(\prod_{k=1}^\infty K_k\right)$ is well defined. Notice also that the product $\prod_{k=1}^\infty K_k$ may not be compact and hence, even if $f$ is continuous, $\overline{f\left(\prod_{k=1}^\infty K_k\right)}$ may not be compact too. 
}
\end{remark}

Eventually, if $\mS=\{f_1,...,f_n\}$ is a family of contractions $f_i: \ell_\infty(X) \to X$ which satisfy (at least) (C1)~condition, then we say that $\mS$ is a \emph{generalized iterated function system of infinite order} (\emph{GIFS$_\infty$} in short). 

\begin{remark}\label{C1eucl}\emph{
Notice that in the case when $X$ is an Euclidean space, any GIFS$_\infty$ fulfills (C1). Let $f: \ell_\infty(\R^n) \to \R^n$ be a contraction. If~\mbox{$(K_k) \in \ell_\infty(\K(\R^n))$} then $\bigcup_{k=1}^\infty K_k \subset \R^n$ is bounded and for any $x, y \in \prod_{k=1}^\infty K_k$ we have
$$d(f(x), f(y)) \leq \on{Lip}(f) d_1(x,y) \leq \on{Lip}(f) \on{diam}\left(\bigcup_{k=1}^\infty K_k\right).$$
Hence $f\left(K_1 \times K_2 \times ... \right)$ is bounded in~$\R^n$ and $\overline{f\left(K_1 \times K_2 \times ...\right)} \in \K(\R^n)$. \\
}
\end{remark}

Secelean obtained the following counterpart of H-B Theorem:
\begin{theorem}\label{se2}
Let $X$ be a complete metric space and $\mS=\{f_1,...,f_n\}$ be a~GIFS$_\infty$. Then there is a~unique $A_\mS\in\K(X)$ such that
$$
A_\mS = \bigcup_{i=1}^n \overline{f_i\left(\prod_{k=1}^\infty A_\mS\right)}.
$$
\end{theorem}
The set $A_\mS$ is called the \emph{fractal} or the \emph{attractor} of GIFS$_\infty$ $\mS$.  


\section{Basic definitions}
Let $(X,d)$ be a metric space and $A, B \subset X$. By $\K(A)$ we denote the subspace of $\K(X)$ which consists of nonempty, compact subsets of $A$. By~$\on{Int} A$ and $\overline{A}$ we denote accordingly the \emph{interior} and the \emph{closure} of~$A$. By~$\on{diam}(A)$ we denote the \emph{diameter of $A$} (we set $\on{diam}(\emptyset) = 0$). For $r>0$ we define \emph{$r$-neighbourhood \mbox{of $A$}} by $B(A, r) := \left\{x \in X: \exists_{a \in A}\ d(a,x) < r \right\}$. The \emph{distance between $A$ and $B$} is defined by \mbox{$\on{dist}(A,B) := \inf\{d(a,b): a \in A, b\in B\}$}. We~say that $A$ is \emph{nowhere dense} if $\on{Int} \overline{A} = \emptyset$ and that it is \emph{meager} if it is a countable union of nowhere dense sets. A set is \emph{co-meager} if its complement is meager. Note that a set is co-meager iff it contains a dense $G_\delta$ set (i.e.~a~dense set which is a countable intersection of open sets). We call $A$ \emph{perfect} if it is closed and does not contain isolated points. We say that $X$ is a~\emph{Polish metric space} if it is separable and complete. Cantor-Bendixson Theorem (\cite[Theorem 6.4]{Kechris1995}) states that a Polish metric space $X$ is a disjoint union of perfect set $P$ and countable open set $U$. When $X$ is complete, we~say that a~\emph{typical element of $X$} has property $\mathcal{P}$ if the set $\{x \in X: x \textrm{ has property } \mathcal{P}\}$ is co-meager. 

Let $\delta >0$ and $A \subset X$. We say that a family $\{U_i\}_{i \in I}$ of subsets of $X$ is a~\emph{$\delta$-cover of $A$} if $I$ is countable, $A \subset \bigcup_{i\in I} U_i$ and $\on{diam}(U_i) \leq \delta$ for $i\in I$.
We~call a function $h : [0,\infty) \to [0,\infty)$ a \emph{gauge function} if it is nondecreasing, right-continuous and $h(x) = 0$ iff $x=0$. If $h$ is a gauge function then we set 
$$\Hd^h_\delta(A) := \inf \left\{\sum_{i\in I} h\left(\on{diam}(U_i)\right) : \{U_i\}_{i \in I} \textrm{ is a $\delta$-cover of $A$} \right\}.$$
Then we put
$$\Hd^h(A) := \lim_{\delta \to 0} \Hd^h_\delta(A).$$
$\Hd^h(A)$ is called \emph{$h$-Hausdorff measure of $A$}. If $0<\Hd^h(A)<\infty$ then we say that $A$ is \emph{$h$-visible}. 

Let us denote by $\N^{\leq n}$ a family of finite sequences of natural numbers of~length at most $n$, by $\N^{<\omega}$ a family of all finite sequences of natural numbers. Also by $2\N+1$ denote the set of odd numbers.

\section{Typical compact set in polish spaces}
In \cite{Balka2013}, \cite{Balka2013a} Balka and M\'ath\'e introduced a construction of ''balanced sets''. We will present this construction with a slight modification which will be crucial for concluding about self-similarity of typical compact sets. During this section we will comment if any significant change in the original proofs should be concerned. We begin with a definition of a $q$-balanced set.

\begin{definition}\emph{
Let $(a_n) \subset \N$. For any $n\in\N$, define 
$$\I_n := \prod_{i=1}^n \{1, ..., a_i\}$$
(i.e. $\I_n$ is a set of sequences of length $n$ with $i$-th element from $\{1,...,a_i\}$) and put 
$$\I := \bigcup_{n=1}^\infty \I_n.$$
We say that $\Phi: 2\N+1 \to \I$ is an \emph{indexing function according to \mbox{sequence $(a_n)$}} if $\Phi$ is surjective and $\Phi(n) \in \bigcup_{i=1}^n \I_i$ for any odd $n$.
}
\end{definition}

\begin{definition}\label{defzbalansowany}\emph{
Let $X$ be a complete metric space and $q\geq 2$. We will say that a compact set \mbox{$K\subset X$} is \emph{$q$-balanced} if
\begin{equation}\label{cpostac}
K = \bigcap_{n\in\N} \left(\bigcup_{i_1=1}^{a_1} \ldots \bigcup_{i_n=1}^{a_n} C_{i_1, ..., i_n} \right),
\end{equation}
where \mbox{$(a_n) \subset \N$} is such that: 
\begin{itemize}
\item[(i)] $a_1 \geq 2$ and $a_{n+1} \geq n a_1 ... a_n$,
\end{itemize}
$C_{i_1, ..., i_n}$ are nonempty, closed and satisfy
\begin{itemize}
\item[(ii)] $C_{i_1, ... ,i_n, i_{n+1}} \subset C_{i_1, ..., i_n}$ for any $(i_1, ...,i_n) \in \I_n$ and $i_{n+1} \in \{1, ..., a_{n+1}\}$,
\end{itemize}
and there exist a sequence $(b_n) \subset [0,\infty)$ and an indexing function \mbox{$\Phi: 2\N+1 \to \I$} according to \mbox{sequence $(a_n)$} such that for any $n\in\N$ and \mbox{$(i_1, ..., i_n), (j_1, ..., j_n) \in \I_n$}:
\begin{itemize}
\item[(iii)] $\on{diam} \left(C_{i_1, ..., i_n}\right) \leq b_n$;
\item[(iv)] $\on{dist}(C_{i_1, ..., i_n}, C_{j_1, ..., j_n}) > q b_n$ whenever $(i_1, ..., i_n) \neq (j_1, ..., j_n)$;
\item[(v)] If $n$ is odd and $C_{i_1, ..., i_n} \subset C_{\Phi(n)}$, and $C_{j_1, ..., j_n} \not\subset C_{\Phi(n)}$ then for any different $s,t\in\{1, ..., a_{n+1}\}$
$$\on{dist}(C_{i_1, ..., i_n, s}, C_{i_1, ..., i_n, t}) > \on{diam} \left( \bigcup_{j_{n+1} = 1}^{a_{n+1}} C_{j_1, ..., j_n, j_{n+1}} \right).$$
\end{itemize}
}
\end{definition}

\begin{remark}\emph{
The mentioned modification bases on concerning $q \geq 2$ factor in Definition~\ref{defzbalansowany}(iv). Since this change does not affect the idea of proofs from \mbox{\cite{Balka2013}, \cite{Balka2013a}}, we will just comment where they can be precisely found. 
}
\end{remark}

\begin{remark}\emph{
Notice that in countable complete metric space there is no $q$-balanced set (since any such set has $\mathfrak{c}$ cardinality).
}
\end{remark}

\begin{theorem}\label{istnienie}
Let $X$ be an uncountable Polish metric space and $q\geq 2$. There exists $q$-balanced \mbox{set $K\subset X$}.
\end{theorem}

\begin{proof}
\cite[Theorem 3.5]{Balka2013}
\end{proof}

\begin{theorem}\label{wsiebie}
Let $X$ be an uncountable Polish metric space and let $K$ be a~$q$-balanced set. Then there exists continuous gauge function $h$ such that $K$ is \mbox{$h$-visible} and for any weak contraction $f:K \to X$ 
$$\Hd^h(K\cap f(K)) = 0.$$
\end{theorem}

\begin{proof}
\cite[Theorem 5.1]{Balka2013}
\end{proof}

\begin{restatable}{theorem}{twzbal}
\label{twzbal}
Let $X$ be a Polish metric space. Typical compact set $K \subset X$ is either finite or it is a~union of finite set and a $q$-balanced set (for any $q \geq 2$). \\
Additionally, if $X$ is perfect then the typical compact set is a $q$-balanced set (for any $q \geq 2$).
\end{restatable}

\begin{remark}\emph{
In the original proof of this theorem (\cite[Theorem 4.5]{Balka2013a}) there is a flaw. In Section~\ref{appendix} we will first present a counterexample for the incorrect statement in the proof and then we will correct the flaw. 
}
\end{remark}

As a consequence the following holds
\begin{corollary}\label{typweak}
Let $X$ be Polish metric space. Typical compact set $K\subset X$ is either finite, or there exists continuous gauge function $h$ such that $K$ is \mbox{$h$-visible} and for any weak contraction $f:K \to X$, \mbox{$\Hd^h(K\cap f(K)) = 0$}.
\end{corollary}

\section{Balanced sets as attractors of GIFSs$_\infty$}
We will now study mentioned notions of self-similarity of typical compact sets in Polish metric spaces. Directly from the results of Balka and M\'ath\'e we get
\begin{theorem}\label{zbalIFS}
Let $X$ be Polish metric space. Then a typical compact set is~either finite or it is not a fractal of any weak IFS.
\end{theorem}

\begin{proof}
If $X$ is countable, then a typical compact set is finite. Assume that $X$ is uncountable. Then by Corollary~\ref{typweak} a typical compact set $K$ is either finite or there exists continuous gauge function $h$ such that $K$ is \mbox{$h$-visible} and for any weak contraction $f:K \to X$, $\Hd^h(K\cap f(K)) = 0$. Suppose that $K$ is~infinite and take $h$ from the second part of statement. Then $\Hd^h(K) > 0$. Suppose that there exists a~weak IFS $\F = \{f_1, ..., f_n\}$ such that $K$ is its fractal, i.e. $K = \bigcup_{i=1}^n f_i(K)$. We have 
$$0 < \Hd^h(K) = \Hd^h\left(\bigcup_{i=1}^n f_i(K)\right) \leq \sum_{i=1}^n \Hd^h\left(K \cap f_i(K)\right) = 0$$
which is a contradiction. Hence $K$ is not a fractal of any weak IFS.
\end{proof}

\begin{remark}\label{topol}\emph{
Note that any $q$-balanced set is a Cantor set and hence it is homeomorphic to ternary Cantor set which is clearly an IFS fractal. Moreover a union of topological fractal and a finite set is a~topological fractal. Therefore in Polish metric spaces a typical compact set is a topological fractal.
}
\end{remark}

\begin{definition}\emph{
We say that $A \in \K(X)$ is a \emph{generalized fractal} if there exists GIFS$_\infty$ $\F :=\{f_1, ..., f_n\}$ defined on $\ell_\infty(A)$ such that $f_i:\ell_\infty(A) \to X$ are contractions and  
$$A = \bigcup_{i=1}^n \overline{f_i(A \times A \times \ldots)}.$$
We will say that $\F$ \emph{witnesses that $A$ is a generalized fractal} (shortly that \emph{$\F$ witnesses}).
}
\end{definition}

\begin{remark}\label{C1automat}\emph{
Observe that if $A \in \K(X)$ and there exist continuous mappings $f_i:\ell_\infty(A) \to A$ such that $A = \bigcup_{i=1}^n f_i(A \times A \times \ldots)$, then each $f_i$ fulfills (C1). \\
\big(For any \mbox{$(K_k) \in \ell_\infty(\mathcal{K}(A))$}, $f_i(K_1 \times K_2 \times ...)  \subset A$ and so $\overline{f_i(K_1 \times K_2 \times ...)} \in \K(X)$.\big)
}
\end{remark}

We will now show that a typical compact set in Polish metric spaces is a generalized fractal. In correspondence with Theorem~\ref{zbalIFS} this result shows that by considering GIFSs$_\infty$ we can describe significantly more sets than using only IFS theory (unless we decide to change metric -- compare Remark~\ref{topol} and notes on topological fractals in the Introduction).

\begin{theorem}\label{fraktalslabego}
Let $X$ be a Polish metric space and let $q \geq 2$. Each $q$-balanced set $K\subset X$ is a~generalized fractal and a GIFS$_\infty$ $\F$ witnessing this is such that \mbox{$\on{Lip}(f) \leq \frac{1}{q}$} for $f \in \F$.  
\end{theorem}

\begin{proof}
Let $K$ be $q$-balanced set and $(a_n), (b_n)$ be proper sequences and $C_{i_1, ..., i_n}$ proper sets from the definition of $K$ (Definition~\ref{defzbalansowany}). We will define mappings from $\ell_\infty(K)$ to $X$. Since $K$ is bounded, $\ell_\infty(K) = \prod_{i=1}^\infty K$. $K$ can be written as 
\begin{equation}\label{cantor}
K = \left\{x_{(i_1, i_2, ...)}: (i_1, i_2, ...) \in \prod_{i=1}^\infty \{1, ..., a_i\} \right\}
\end{equation}
where for any $(i_1, i_2, ...) \in \prod_{i=1}^\infty \{1, ..., a_i\}$, $x_{(i_1, i_2, ...)}$ is the unique element of~$\bigcap_{n\in\N} C_{i_1, ..., i_n}$. 

Define $\mathcal{A} := \prod_{i=1}^\infty \{1, ..., a_i\}$. For any $i\in \{1,...,a_1\}$ define mapping \mbox{$f_i: \ell_\infty(K) \to K$} by
$$f_i(x_{\alpha_1}, x_{\alpha_2}, ...) := x_{i(\alpha_1, \alpha_2, ...)}$$
where $x_{\alpha_1}, x_{\alpha_2}, ... \in K$ and $\alpha_j = (\alpha_j^{(1)}, \alpha_j^{(2)}, ...) \in \mathcal{A}$ for any $j \in \N$ and 
$$i(\alpha_1, \alpha_2, ...) := (i, \beta_1, \beta_2, ...)$$
where $\beta_j := 1+ \sum_{i=1}^{a_{j+1}-1} \tilde{\alpha}_i^{(j)}$ for $j\geq 1$ and $\tilde{\alpha}_i^{(j)} := \alpha_i^{(j)} \on{mod} 2$ for $i, j\in \N$. 

Fix $i \in \{1, ..., a_1\}$. We now show that $\on{Lip}_{d}(f_i) \leq \frac{1}{q}$. 
Let $(x_{\alpha_1}, x_{\alpha_2}, ...)$, \mbox{$(x_{\beta_1}, x_{\beta_2}, ...) \in \prod_{i=1}^\infty K$}. Assume that $f_i(x_{\alpha_1}, x_{\alpha_2}, ...) \neq f_i(x_{\beta_1}, x_{\beta_2}, ...)$.
\begin{align*}
&d\left(f_i(x_{\alpha_1}, x_{\alpha_2}, ...), f_i(x_{\beta_1}, x_{\beta_2}, ...) \right) = d\left(x_{i(\alpha_1, \alpha_2, ...)}, x_{i(\beta_1, \beta_2, ...)} \right) = \\
&= d\left(x_{(i, 1+\sum_{i=1}^{a_{2}-1} \tilde{\alpha}_i^{(1)}, 1+\sum_{i=1}^{a_{3}-1} \tilde{\alpha}_i^{(2)}, ...)}, x_{(i, 1+\sum_{i=1}^{a_{2}-1} \tilde{\beta}_i^{(1)}, 1+\sum_{i=1}^{a_{3}-1} \tilde{\beta}_i^{(2)}, ...)} \right).
\end{align*}
Let 
$$\eta := \min\left\{k\in\N: \sum_{i=1}^{a_{k+1}-1} \tilde{\alpha}_i^{(k)} \neq \sum_{i=1}^{a_{k+1}-1} \tilde{\beta}_i^{(k)}\right\},$$ 
i.e. $\eta$ indicates the first coordinate on which $i(\alpha_1, \alpha_2, ...)$ and $i(\beta_1, \beta_2, ...)$ differs. \\
If $\eta = 1$ we have $d\left(f_i(x_{\alpha_1}, x_{\alpha_2}, ...), f_i(x_{\beta_1}, x_{\beta_2}, ...) \right) \leq \on{diam}\left(C_{i} \right) \leq b_1 = b_\eta$.
Otherwise
\begin{align*}
&d\left(f_i(x_{\alpha_1}, x_{\alpha_2}, ...), f_i(x_{\beta_1}, x_{\beta_2}, ...) \right) \leq \on{diam}\left(C_{i, 1+\sum_{i=1}^{a_{2}-1} \tilde{\alpha}_i^{(1)}, ..., 1+\sum_{i=1}^{a_{\eta}-1} \tilde{\alpha}_i^{(\eta-1)}} \right) \leq b_\eta.
\end{align*}
We now estimate $d_{1}\left((x_{\alpha_1}, x_{\alpha_2}, ...), (x_{\beta_1}, x_{\beta_2}, ...) \right)$. For any $j\in \N$ define
$$\eta_j := \min\{k\in\N: \alpha_j^{(k)} \neq \beta_j^{(k)}\} \ \ \ \textrm{(assume $\min \emptyset := \infty$)}$$
i.e. for each $j$, $\eta_j$ is the first coordinate on which $\alpha_j$ and $\beta_j$ differs. Denote also
$$l := \min\{k\in\N: \min\{\eta_1, ..., \eta_{a_{k+1}-1}\} \leq k\},$$
i.e. $l$ is the first coordinate which can impact on the difference between images of $(x_{\alpha_1}, x_{\alpha_2}, ...)$ and $(x_{\beta_1}, x_{\beta_2}, ...)$. Notice that $l \leq \eta$. (Since $\sum_{i=1}^{a_{\eta+1}-1} \tilde{\alpha}_i^{(\eta)} \neq \sum_{i=1}^{a_{\eta+1}-1} \tilde{\beta}_i^{(\eta)}$, there exists some \mbox{$1\leq j\leq a_{\eta+1}-1$} such that $\alpha_j^{(\eta)} \neq \beta_j^{(\eta)}$. This implies $\eta_j \leq \eta$ and therefore $l \leq \eta$.) Additionally put 
$$h := \min\{k\in\{1,...,a_{l+1}-1\}: \eta_k \leq l\}.$$
Clearly $\eta_h \leq l$. We get
\begin{align*}
&d_{1}\left((x_{\alpha_1}, x_{\alpha_2}, ...), (x_{\beta_1}, x_{\beta_2}, ...) \right) = \sup_{n\in\N} d(x_{\alpha_n}, x_{\beta_n}) \geq d(x_{\alpha_h}, x_{\beta_h}) \geq \\
&\geq \on{dist}\left(C_{\alpha_h^{(1)}, ..., \alpha_h^{(\eta_h)}}, C_{\beta_h^{(1)}, ..., \beta_h^{(\eta_h)}}\right) \geq q b_{\eta_h} \geq q b_l \geq q b_\eta.
\end{align*}
Hence 
\begin{align*}
d\left(f_i(x_{\alpha_1}, x_{\alpha_2}, ...), f_i(x_{\beta_1}, x_{\beta_2}, ...) \right) &\leq b_\eta = \frac{1}{q} q b_\eta \leq \frac{1}{q} d_{1}\left((x_{\alpha_1}, x_{\alpha_2}, ...), (x_{\beta_1}, x_{\beta_2}, ...) \right).
\end{align*}
Therefore $\on{Lip}(f_i) \leq \frac{1}{q}$. 


It is easy to see that $K = \bigcup_{i=1}^{a_1} f_i(K \times K \times ...)$. Indeed, if $x_{(i_1, i_2, ...)} \in K$ for some \mbox{$(i_1, i_2, ...) \in \mathcal{A} = \prod_{i=1}^\infty \{1, ..., a_i\}$} then we can take sequence \mbox{$\alpha = (\alpha_1, \alpha_2, ...) \in \ell_\infty(\mathcal{A})$} such that for any $k \in \N$:
$$\alpha_j^{(k)} := \left\{ \begin{array}{ll} 1, & j=1,...,i_{k+1}-1, \\
2, & j \geq i_{k+1} \end{array} \right.$$   
and $f_{i_1}(x_{\alpha_1}, x_{\alpha_2}, ...) = x_{(i_1, i_2, ...)}$. By Remark~\ref{C1automat}, $\F := \{f_1, ..., f_{a_1}\}$ fulfills~(C1). 
\end{proof}


\begin{remark}\label{r}\emph{
We can obtain somewhat stronger result -- namely for any $r>0$, there exists a witnessing GIFS$_\infty$ $\F$ such that \mbox{$\on{Lip}(f) \leq r$} for $f \in \F$. We choose $p$ so that $q^{-p} < r$ and consider GIFS$_\infty$ $\left\{f_{(i_1, ..., i_p)}: (i_1, ..., i_p) \in \prod_{i=1}^p \{1, ..., a_i\} \right\}$ given by (we preserve former notation)
$$f_{(i_1, ..., i_p)}  (x_{\alpha_1}, x_{\alpha_2}, ...) := x_{(i_1, ..., i_p, \beta_1, \beta_2, ...)}$$
where $\beta_j := 1+ \sum_{i=1}^{a_{j+p}-1} \tilde{\alpha}_i^{(j)}$ for $j \in \N$. Then $\on{Lip}(f) \leq q^{-p} < r$ for any $f \in \F$.
}  
\end{remark}

\begin{remark}\emph{
GIFS $\F$ defined in the proof of Theorem~\ref{fraktalslabego} (and also in Remark~\ref{r}) is such that any $f_i\in \F$ fulfills (C2), i.e. for any \mbox{$(K_k) \in \ell_\infty(\mathcal{K}(K))$}, $f_i(K_1 \times K_2 \times ...) \in \K(K)$. It is enough to prove that for any \mbox{$(K_k) \in \ell_\infty(\mathcal{K}(K))$} the image \mbox{$f_i(K_1 \times K_2 \times ... ) \subset K$} is closed. 
\newline
\indent Observe that $\left(x_n\right) = \left(x_{(i_1^{(n)}, i_2^{(n)}, ...)}\right)_{n\in\N} \subset K$ converges to some \mbox{$x_{(i_1, i_2, ...)} \in K$} iff for any $k\in\N$ the sequence $(i_k^{(n)})_n \subset \{1, ..., a_k\}$ is eventually constant (this follows from Definition~\ref{defzbalansowany}(iv)). 
\newline
\indent Let \mbox{$(K_k) \in \ell_\infty(\mathcal{K}(K))$} and set $f:=f_i$. We will show that \mbox{$f(K_1 \times K_2 \times ... ) \subset K$} is closed. Let $(x_n) = \left(x_{(i, i_2^{(n)}, ...)}\right) \subset f(K_1 \times K_2 \times ... )$ be convergent to some $x \in K$. For any $n\in\N$ there exists a~sequence $\left(y_1^{(n)}, y_2^{(n)}, ...\right) \in \prod_{i=1}^\infty K_i$ such that $f\left(y_1^{(n)}, y_2^{(n)}, ...\right) = x_n$. Consider addresses of $y_i^{(n)}s$, i.e. sequences \mbox{$\left(i_1^{[i, (n)]}, i_2^{[i, (n)]}, ...\right) \in \prod_{i=1}^\infty \{1, ..., a_i\}$} such that $y_i^{(n)} = x_{\left(i_1^{[i, (n)]}, i_2^{[i, (n)]}, ...\right)} \in K$. 
\newline
\indent Notice that in a multiset $A_1 := \left\langle\left(i_1^{[1, (n)]}, ..., i_{1}^{[a_2-1, (n)]} \right): n\in\N \right\rangle$ (multiset of tuples of first coordinates of addresses of $y_1^{(n)}, ..., y_{a_2-1}^{(n)}$) some $(a_2-1)$-tuple appears infinitely many times (since a tuple $\left(i_1^{[1, (n)]}, ..., i_{1}^{[a_2-1, (n)]} \right)$ belongs to finite set $\{1, ..., a_1\}^{a_2-1}$). Denote one of them by $\left(j_1^{(1)}, ..., j_{a_2-1}^{(1)}\right)$. 
\newline
Now consider only these $n\in\N$ for which the sequence $\left(y_1^{(n)}, y_2^{(n)}, ...\right)$ is such that for any $i=1,...,a_{2}-1$, $i_1^{[i, (n)]} = j_i^{(1)}$ (i.e. for $i=1,...,a_2-1$ the address of $y_i^{(n)}$ begins with $j_i^{(1)}$). Reenumerate obtained subsequence. 
\newline
\indent Assume that for some $k\in\N$ we have a sequence 
$$\left(\left(j_1^{(1)}, ..., j_1^{(k)}\right), ..., \left(j_{a_{k+1}-1}^{(1)}, ..., j_{a_{k+1}-1}^{(k)}\right)\right) \in \left(\prod_{j=1}^{k} \{1, ..., a_j\} \right)^{a_{k+1}-1}$$ 
and infinite number of sequences $\left(y_1^{(n)}, y_2^{(n)}, ...\right)$ such that for any $i=1,...,a_{{k+1}-1}$ and any $n\in\N$, $\left(i_1^{[i, (n)]}, ..., i_k^{[i, (n)]}\right) = \left(j_i^{(1)}, ..., j_i^{(k)}\right)$ \big(i.e. for any $n\in\N$ and $i=1, ..., a_{k+1}-1$, address of $y_i^{(n)}$ begins with \mbox{$\left(j_i^{(1)}, ..., j_i^{(k)}\right)$}\big). Consider a~multiset 
$$A_{k+1} := \left\langle\left((i_1^{[1, (n)]}, ..., i_{k+1}^{[1, (n)]}), ..., (i_{1}^{[a_{k+2}-1, (n)]}, ..., i_{k+1}^{[a_{k+2}-1, (n)]}) \right): n\in\N \right\rangle.$$
One of the $(a_{k+2}-1)$-tuples of $(k+1)$-tuples $(i_1^{[\cdot]}, ..., i_{k+1}^{[\cdot]})$ appears infinitely many times. We denote one of them by $\left((j_1^{(1)}, ..., j_1^{(k+1)}), ..., (j_{a_{k+2}-1}^{(1)}, ..., j_{a_{k+2}-1}^{(k+1)})\right)$. From the original sequences we leave only these $\left(y_1^{(n)}, y_2^{(n)}, ...\right)$ such that for any $i=1,...,a_{k+2}-1$ and any $n\in\N$, $\left(i_1^{[i, (n)]}, ..., i_{k+1}^{[i, (n)]}\right) = \left(j_i^{(1)}, ..., j_i^{(k+1)}\right)$. Then we reenumerate obtained subsequence.
\newline
\indent Observe that during this procedure, the sequence of bigger and bigger tuples $(j_i^{(1)}, ..., j_i^{(k)})_k$ is growing consistently -- in the sense that $j_i^{(1)}, ..., j_i^{(k)}$ are the first $k$-elements of $(j_i^{(1)}, ..., j_i^{(k+1)})$ for any $i,k\in\N$.
\newline
\indent Eventually, observe that for any $n\in\N$ the set $\left(\bigcap_{i\in\N} C_{(j_n^{(1)}, ..., j_n^{(i)})} \right) \cap K_n$ is nonempty. Fix $n\in\N$. Clearly $\bigcap_{i\in\N} C_{(j_n^{(1)}, ..., j_n^{(i)})} = x_{(j_n^{(1)}, j_n^{(2)}, ...)}$. Moreover $K_n$ contains $x_{(j_n^{(1)}, j_n^{(2)}, ...)}$. This comes from the fact that among sequence $(y_i^{(n)})_i \subset K_n$ we can find $z_1$ whose address begins with $j_n^{(1)}$, then $z_2$ whose address begins with $(j_n^{(1)}, j_n^{(2)})$ and so on. Hence we obtain sequence of elements $z_i \in K_n \cap C_{(j_n^{(1)}, ..., j_n^{(i)})}, i\in\N$, which is convergent \big(since $\on{diam}\left(C_{(j_n^{(1)}, ..., j_n^{(i)})} \right) \stackrel{i\to\infty}{\to} 0$\big). Let us denote this limit by $z^{(n)}$. Clearly, $z^{(n)}$ belongs to $C_{(j_n^{(1)}, j_n^{(2)}, ...)}$ and its address is $(j_n^{(1)}, j_n^{(2)}, ...)$. Also, since $K_n \in \mathcal{K}(K)$ is closed, $z^{(n)}$ belongs to~$K_n$. 
\newline
\indent Finally, notice that $f\left(z^{(1)}, z^{(2)}, ... \right) = x$. Let $(i, m_1, m_2, ...)$ be the address of $x$. We need to ensure that for any $k\in\N$, $\sum_{i=1}^{a_{k+1}-1} j_i^{(k)} = m_k$. This follows from the fact that since $x_n \to x$, there are some moments from which the consecutive coordinates in addresses of $x_n$s and $x$ match. This means that sequences $\left(y_1^{(n)}, y_2^{(n)}, ...\right)$ are such that their addresses fulfill $\sum_{i=1}^{a_{k+1}-1} i_k^{[i, (n)]} = m_k$ for $n$ big enough. This is exactly what $(j_1^{(k)}, ..., j_{a_{k+1}-1}^{(k)})$ fulfills (otherwise it could not appear infinitely many times during construction). 
}
\end{remark}


\begin{corollary}
Let $X$ be a perfect Polish metric space. Typical compact set $K\subset X$ is a generalized fractal such that for any $r\in (0,1)$ there exists a~GIFS$_\infty$ $\F$ witnessing this with \mbox{$\on{Lip}(f) \leq r$} for $f \in \F$.  
\end{corollary}

\begin{proof}
Combine Theorem~\ref{twzbal} and Remark~\ref{r}.
\end{proof}

In case of Euclidean spaces this result can be easily extended in~order to~guarantee that a typical compact set in $\R^n$ is GIFS$_\infty$' fractal (we consider Euclidean metric on $\R^n$). 

\begin{proposition}[{\cite[Lemma 1.1]{Benyamini2000}}]~\\
Let $X$ be a metric space, $A \subset X$ and $f: A \to \R$ be such that $\on{Lip}(f)<\infty$. Then there exists $\tilde{f}: X \to \R$ such that $\tilde{f}|_A = f$ and $\on{Lip}(\tilde{f}) = \on{Lip}(f)$.
\end{proposition} 

\begin{corollary}\label{rozszerzenie}
Let $n\in\N$, $X$ be a metric space, $A \subset X$ and $f: A \to \R^n$ be such that $\on{Lip}(f)<\infty$. Then there exists $\tilde{f}: X \to \R^n$ such that $\tilde{f}|_A = f$ and $\on{Lip}(\tilde{f}) \leq \sqrt{n}\on{Lip}(f)$.
\end{corollary}

\begin{theorem}\label{euklides}
Let $n\in\N$, $q \geq 2$ and $K\subset \R^n$ be a $q$-balanced set. For any $r\in (0,1)$, $K$ is a fractal of some GIFS$_\infty$ $\F$ such that mappings $f\in \F$ fulfill \mbox{$\on{Lip}(f) \leq r$}. 
\end{theorem}

\begin{proof} 
Let $r\in (0,1)$. By Remark~\ref{r}, there exists GIFS$_\infty$ $\F:=\{f_1, ..., f_p\}$ defined on $\ell_\infty(K)$ witnessing that $K$ is a generalized fractal and such that $\on{Lip}(f) \leq\frac{1}{\sqrt{n}} r$ for $i=1,...,p$.  \\
From Corollary~\ref{rozszerzenie} for each $f_i$ there exists extension $\tilde{f}_i: \ell_\infty(\R^n) \to \R^n$ with $\on{Lip}(\tilde{f}_i) \leq \sqrt{n} \on{Lip}(f_i) \leq r$. By Remark~\ref{C1eucl} each $\tilde{f}_i$ satisfies (C1). Hence $K$ is~a~fractal of the GIFS$_\infty$ $\{\tilde{f}_1, ..., \tilde{f}_{p}\}$.
\end{proof}

\begin{corollary}
Typical compact set $K \subset \R^n$ is such that for any $r\in (0,1)$ it is a fractal of some GIFS$_\infty$~$\F$ with \mbox{$\on{Lip}(f) \leq r$} for $f \in \F$.
\end{corollary}

We now generalize Theorem~\ref{fraktalslabego} to cover case of sets which are union of~balanced set and~a~finite set. 

\begin{theorem}\label{union}
Let $X$ be a Polish metric space and $q \geq 2$. Let $R = K \cup P$ be a union of $q$-balanced set~$K$ and~finite set $P$ with $K \cap P = \emptyset$. Then $R$ is~a~generalized fractal.
\end{theorem}

\begin{proof}
Let $\eta := \on{dist}(K, P) > 0$ and $r \in (0,1)$. By Remark~\ref{r} there exists GIFS$_\infty$ $\F := \{f_1, ..., f_p\}$ witnessing that $K$ is a generalized fractal (in particular $f_i$ are defined on $\ell_\infty(K)$) with $\on{Lip}(f_i) \leq r  \cdot \min \left\{1, \frac{\eta}{\on{diam}(K)} \right\}$. Fix some $x \in K$. For any $i=1,..., p$ define $\tilde{f}_i: \ell_\infty(R) \to R$ by 
$$\tilde{f}_i(x_1, x_2, ...) := \left\{ \begin{array}{lll} f_i(x_1, x_2, ...) & \textrm{ if } & (x_1, x_2, ...) \in \ell_\infty(K), \\
f_i(x, x, ...)& \textrm{ if } & (x_1, x_2, ...) \notin \ell_\infty(K). \end{array} \right.$$
Then $\on{Lip}(\tilde{f}_i) \leq  \on{Lip}{f_i} \cdot \max \left\{1, \frac{\on{diam}(K)}{\eta}\right\} \leq r$. Indeed, let \mbox{$(x_1, x_2, ...), (y_1, y_2, ...) \in \ell_\infty(R)$}. Consider cases:
\begin{itemize}
\item[1)] $(x_1, x_2, ...), (y_1, y_2, ...) \in \ell_\infty(K)$. Then 
\begin{align*}
d\left(\tilde{f}_i(x_1, x_2, ...), \tilde{f}_i(y_1, y_2, ...) \right) &= d\left(f_i(x_1, x_2, ...), f_i(y_1, y_2, ...) \right) \leq \\
&\leq \on{Lip}(f_i) d_1 ((x_1, x_2, ...), (y_1, y_2, ...)).
\end{align*}
\item[2)] $(x_1, x_2, ...), (y_1, y_2, ...) \notin \ell_\infty(K)$. Then $d\left(\tilde{f}_i(x_1, x_2, ...), \tilde{f}_i(y_1, y_2, ...) \right) = 0$.
\item[3)] $(x_1, x_2, ...) \in \ell_\infty(K)$ and $(y_1, y_2, ...) \notin \ell_\infty(K)$. Since $(y_1, y_2, ...) \notin \ell_\infty(K)$ there exists $k\in\N$ such that $y_k \in P$. Then
\begin{align*}
&d\left(\tilde{f}_i(x_1, x_2, ...), \tilde{f}_i(y_1, y_2, ...) \right) = d\left(f_i(x_1, x_2, ...), f_i(x, x, ...) \right) \leq \\
&\;\;\;\;\;\;\;\; \leq \on{Lip}(f_i) d_1 ((x_1, x_2, ...), (x, x, ...)) \leq \on{Lip}(f_i) \on{diam}(K) = \\
&\;\;\;\;\;\;\;\; = \frac{\on{Lip}(f_i) \on{diam}(K)}{\eta} \cdot \eta \leq \frac{\on{Lip}(f_i) \on{diam}(K)}{\eta} \cdot d(x_k, y_k) \leq \\
&\;\;\;\;\; \leq  \frac{\on{Lip}(f_i) \on{diam}(K)}{\eta} d_1((x_1, x_2, ...), (y_1, y_2, ...)).
\end{align*}
\end{itemize}
Additionally, for any $x \in P$ define constant mappings $h_x: \ell_\infty(R) \to R$ by~\mbox{$h_x \equiv x$}. Then $\tilde{F} := \F \cup \{h_x: x \in P\}$ is a GIFS$_\infty$ defined on $\ell_\infty(R)$ and its mappings are such that $\on{Lip}(g) \leq r$ for any $g \in \tilde{F}$. Moreover
$$R = K \cup P = \bigcup_{f \in \F} f(R \times R \times ...) \cup \bigcup_{x \in P} h_x(R \times R \times ...).$$
Therefore $R$ is a generalized fractal and $\tilde{F}$ witnesses this.
\end{proof}

\begin{corollary}
Let $X$ be a Polish metric space. Typical compact set $K\subset X$ is a generalized fractal such that for any $r\in (0,1)$ there exists a GIFS$_\infty$ $\F$ witnessing this with \mbox{$\on{Lip}(f) \leq r$} for $f \in \F$.  
\end{corollary}

\section{Appendix: correction of the proof of Theorem~\ref{twzbal} (\cite[Theorem 4.5]{Balka2013a})}\label{appendix}
We will correct proof of the following theorem

\twzbal*

First we shall see where is the flaw. Let $X^* \subset X$ be a perfect subset of $X$ such that $X \setminus X^*$ is countable open ($X^*$ exists by Cantor-Bendixon Theorem). From the first part of proof (when considering perfect space) one~get a~family $\F^*$ which is dense $G_\delta$ in~$\K(X^*)$ and which consists of $q$-balanced sets. In order to show that the typical compact set is a union of finite set and $q$-balanced set the authors consider the mapping $R: \K(X) \to \K(X^*) \cup \{\emptyset\}$ given by
$$R(K) := K \cap X^*, \ \ \ \ K \in \K(X)$$
(put $H^d(\emptyset, A) = 1$ whenever $A \in \K(X^*)$; then $\emptyset$ is an isolated point in~\mbox{$\K(X^*)\cup\{\emptyset\}$}). \\
They show that $R$ is open and claim that $\tilde{\F} := R^{-1}(\F^* \cup \{\emptyset\})$ is dense $G_\delta$ in~$\K(X)$. However, it may not be the case since $R$ is not continuous and $\tilde{\F}$ may not be $G_\delta$. Consider following example. (We will repair the flaw afterward).

\begin{example}\emph{
Let 
$$X := \big(\{0\} \times ([0,1] \cup [2,3])\big) \cup \bigcup_{n=1}^\infty \left\{\left(\frac{1}{n}, \frac{i}{n}\right): i=0, ..., n, 2n, ..., 3n \right\}.$$
Then \mbox{$X = X^* \cup S$} where $X^* = \{0\} \times ([0,1] \cup [2,3])$ and $S = X \setminus X^*$ is a set of isolated points. $R$~is not continuous since the sequence $K_n := \{(0,0)\} \cup \left\{\left(\frac{1}{n}, \frac{i}{n}\right): i=0, ..., n \right\}$ tends to $K:= \{0\} \times [0,1]$, but $R(K_n) = \{(0,0)\}$ and $R(K) = K$. 
}
\end{example}

\begin{proof} \emph{(of Theorem~\ref{twzbal}):} 
The first assertion of the proof (when perfect space is considered) holds (\cite[Theorem 4.5]{Balka2013a}). 

Let now $X$ be a separable, complete metric space, $X^* \subset X$ be a perfect subset of $X$ such that $U:=X \setminus X^*$ is countable open. By $S$ let us denote the set of isolated points of $X$. Then $S \subset U$, \mbox{$S$ is open} and $U \subset \overline{S}$. Since $S$ is open and dense in $\overline{S}$, $\K(S)$ is open and dense in $\K(\overline{S})$. Moreover, any~element of~$\K(S)$ is finite.


From the first part of proof there exists a family $\F^*$ which is a dense $G_\delta$ set in $\K(X^*)$ and which consists of $q$-balanced sets. 

Let
$$\F := \left\{K \in \K(X): \ K \cap \overline{S} \subset S \textrm{ and } K \cap X^* \in \F^* \cup \{\emptyset\} \right\}.$$
Then any $K \in \F$ is a union of emptyset or $q$-balanced set and finite set consisting of isolated points. We will show that $\F$ contains a dense $G_\delta$ set. 

If $X^* \setminus \overline{S} = \emptyset$ then $\K(S)$ is open and dense in $\K(\overline{S}) = \K(X)$, so~$\K(S)$ is~comeager in $\K(X)$.

Assume that $X^* \setminus \overline{S} \neq \emptyset$. There are dense and open $\F_n \subset \K(X^*)$ such that $\F^* = \bigcap_{n\in\N} \F_n$. For any $n\in\N$ define
$$\F_n' := \left\{A \cup B \neq \emptyset: \ A \in (\F_n \cap \K(X^* \setminus \overline{S}))\cup\{\emptyset\} \ \textrm{ and } \ B\in \K(S)\cup\{\emptyset\} \right\}.$$

Let $K \in \bigcap_{n\in\N} \F_n'$. Then, since $K \in \F_1'$, $K$ is nonempty and $K = A \cup B$ where $A \in \F_1 \cap \K(X^* \setminus \overline{S}))\cup\{\emptyset\}$ and $B \in \K(S) \cup \{\emptyset\}$. If $A = \emptyset$ then $K = B \in \K(S) \subset \F$. Assume that $A \neq \emptyset$. Then $K \in \bigcap_{n\in\N} \F_n'$ implies $A \in \F_n$ for any $n\in\N$. Therefore $K \in \F$.

We~now show that $\F_n'$ are dense and open in~$\K(X)$. 
Let $n\in\N$. We~first show openness. Let $K \in \F_n'$. Assume that $K = A \cup B$ where \mbox{$A \in \F_n \cap \K(X^* \setminus \overline{S})$} and $B \in \K(S)$ (in rest of cases we conduct analogously). Since $A \in \K(X^* \setminus \overline{S})$ and $X^* \setminus \overline{S}$ is open, there exists $r_1 > 0$ such that $B(A, r_1) \subset X^* \setminus \overline{S}$. Since $\F_n$ is open in $\K(X^*)$, there exists $r_2>0$ such that~$B_H(A, r_2) \cap \K(X^*) \subset \F_n$. For $B \in \K(S)$, which is a finite union of isolated points in $X$, there exists $r_3>0$ such that \mbox{$B_H(B, r_3) = \{B\}$} (in particular $B(B,r_3) \subset B$). Then for $r:=\frac{1}{2}\min\{r_1, r_2, r_3\}$ we have
$$B_H(A\cup B, r) \subset \F_n'.$$
Indeed, let $F \in B_H(A \cup B, r)$. Notice that $B(A \cup B, r) = B(A, r) \cup B(B,r)$. Hence, since $F \in B_H(A \cup B, r)$, we have $F \subset B(A \cup B, r) = B(A,r) \cup B$. Define $C:=F \cap B(A,r)$ and $D:=F \cap B$. Then $C, D$ are disjoint and $C \in \K(X^*\setminus \overline{S})$ and $D \in \K(S) \cup \{\emptyset\}$. 
It remains to show that $C \in \F_n$. Since~\mbox{$A \cup B \subset B(F, r) = B(C, r) \cup B(D, r)$} and 
$$A \cap B(D, r) \subset A \cap B(B, 2r) \subset A \cap B = \emptyset,$$ 
 we have $A \subset B(C,r)$. Therefore $H(A, C) < r$, so \mbox{$C \in B_H(A, r) \cap \K(X^*) \subset \F_n$}. Eventually \mbox{$F = C \cup D \in \F_n'$}.

We show denseness. Let $K \in \K(X)$ and $\varepsilon>0$. Since a family of finite sets is dense in $\K(X)$, there exists $L = \{x_1, ..., x_k\} \in \K(X)$ such that~\mbox{$H(K, L) < \frac{\varepsilon}{2}$}. Consider $L_1 := L \cap (X^* \setminus \overline{S})$ and $L_2 := L \cap \overline{S}$. Assume that $L_1, L_2$ are nonempty (other cases are analogous). There exists $r>0$ such that $B_H(L_1, r) \subset \K(X^* \setminus \overline{S})$. Then $U:=B_H(L_1, \min\left\{\frac{\varepsilon}{4}, r\right\}) \subset \K(X^* \setminus \overline{S}) \subset \K(X^*)$ and since $\F_n \subset \K(X^*)$ is dense, therefore there is \mbox{$A \in \F_n \cap U \subset \F_n \cap \K(X^* \setminus \overline{S})$}. In particular $H(A, L_1) < \frac{\varepsilon}{4}$. On the other hand, $L_2 \subset \overline{S}$ is finite and for any $x \in L_2$ there exists $x' \in S$ such that $d(x, x') < \frac{\varepsilon}{4}$. Then $B := \{x': x \in L_2\} \in \K(S)$ belongs to $B_H(L_2, \frac{\varepsilon}{4}) \subset \K(S)$. We have $A \cup B \in \F_n'$ and
\begin{align*}
H(A \cup B, K) &\leq H(A \cup B, L) + H(L, K) \leq H(A, L_1) + H(B, L_2) + H(L,K) < \varepsilon.
\end{align*}
\end{proof}

\section*{Acknowledgments} 
\noindent Research is supported by \emph{Narodowe Centrum Nauki}, grant number 2017/26/A/ST1/00189. \\
The author would like to thank Filip Strobin for many discussions and valuable remarks.

\bibliography{fraktale}


\end{document}